\documentclass[12pt,reqno]{amsart}
\let\oldlabel=\label
\def\prellabel{\marginparsep=1em\marginparwidth=44pt
  \def\label##1{\oldlabel{##1}\ifmmode\else\ifinner\else
         \marginpar{{\footnotesize\ \\ \tt
                    ##1}}\fi\fi}}
\def\NZQ{\mathbb}               
\def\NN{{\NZQ N }}

\def\ZZ{{\NZQ Z}}

\def\opn#1#2{\def#1{\operatorname{#2}}} 
\opn\chara{char}
\opn\rank{rank}
\opn\hilb{Hilb}
\opn\reg{reg}
\opn\projdim{projdim}
\opn\Ass{Ass}
\opn\gr{gr}
\opn\Rees{{\mathcal R}}
\let\sect=\cap

\let\tensor=\otimes

\let\Sect=\bigcap
\let\Dirsum=\bigoplus

\usepackage[latin1]{inputenc}
\usepackage{mathptmx}
\newtheorem{theorem}{Theorem}[section]
\newtheorem{lemma}[theorem]{Lemma}
\newtheorem{corollary}[theorem]{Corollary}
\newtheorem{proposition}[theorem]{Proposition}

\theoremstyle{definition}
\newtheorem{remark}[theorem]{Remark}
\newtheorem{definition}[theorem]{Definition}
\newtheorem{example}[theorem]{Example}
\newtheorem{remark/example}[theorem]{Remark/Example}

\newtheorem{question}[theorem]{Question}

\newtheorem*{theorem*}{Theorem}

\let\epsilon=\varepsilon
\let\phi=\varphi
\let\kappa=\varkappa

\opn\ini{in}
\opn\KRS{KRS}
\opn\krs{krs}
\opn\Krs{Krs}
\opn\DEL{DEL}
\opn\diag{diag}
\opn\Ker{Ker}
\opn\Image{Im}
\opn\DD{{\mathcal D}}
\opn\SS{{\mathcal S}}
\opn\MM{{\mathcal M}}
\opn\GL{GL}
\opn{\hht}{ht}
\opn\Cl{Cl}
\opn\cl{cl}
\opn\height{height}
\opn\id{id}
\opn\Hom{Hom}
\opn\Tor{Tor}
\opn\Ext{Ext}
\opn\Spec{Spec}

\def\cF{{\mathcal F}}

\def\cS{{\mathcal S}}
\def\cI{{\mathcal I}}

\def\cP{{\mathcal P}}
\def\cL{{\mathcal L}}

\def\mm{{\mathfrak m}}

\def\addots{\mathinner{\mkern1mu\raise1pt\hbox{.}\mkern2mu\raise4pt\hbox{.}
        \mkern2mu\raise7pt\vbox{\kern7pt\hbox{.}}\mkern1mu}}

\def\discuss#1{\relax}

\numberwithin{equation}{section}

\author{Winfried Bruns}
\address{Universit\"at Osnabr\"uck, Institut f\"ur Mathematik, 49069 Osnabr\"uck, Germany}
\email{wbruns@uos.de}

\author{Aldo Conca}
\address{ Dipartimento di Matematica,
Universit\`a degli Studi di Genova, Italy}
\email{conca@dima.unige.it}

\title{Linear resolutions of powers and products}

\dedicatory{To Gert-Martin Greuel on the occasion of his seventieth birthday}


\date{}

\begin{document}

\begin{abstract}
The goal of this paper is to present examples of  families   of homogeneous ideals in the polynomial ring over a field that satisfy the following condition:  every  product of ideals of  the family has a linear free resolution. As we will see, this condition is strongly correlated to good  primary decompositions of the products and good homological and arithmetical properties of the associated multi-Rees algebras. The following families will be discussed in detail: polymatroidal ideals, ideals generated by linear forms and Borel fixed ideals of maximal minors. The main tools are Gröbner bases and Sagbi deformation.  
\end{abstract}

\maketitle


\section{Introduction}
The goal of this paper is to present examples of  families   of homogeneous ideals in the polynomial ring $R=K[X_1,\dots, X_n]$  over a field $K$ that satisfy the following condition:  every  product of ideals of  the family has a linear free resolution. As we will see, this condition is strongly correlated to ``good''  primary decompositions of the products and ``good'' homological and arithmetical properties of the associated multi-Rees algebras.

For the notions and basic theorems of commutative algebra that we will use in the following we refer the reader to the books of Bruns and Herzog \cite{BH} and Greuel and Pfister \cite{GP}. However, at one point our terminology differs from \cite{GP}: where \cite{GP} uses the attribute ``lead'' (monomial, ideal etc.) we are accustomed to ``initial'' and keep our convention in this note. Moreover,\discuss{Added} we draw some standard results from our survey \cite{BC2}.

The extensive experimental work for this paper (and its predecessors) involved several computer algebra systems: CoCoA \cite{Cocoa}, Macaulay 2 \cite{M2}, Normaliz \cite{Nmz} and Singular \cite{DGPS}.

Let us first give a name to the condition on free resolutions in which we are interested and as well recall the definition of ideal with linear powers from \cite{BCV}. 

\begin{definition}\leavevmode
\begin{enumerate}
\item A homogeneous ideal $I$ of $R$ has \emph{linear powers} if $I^k$ has a linear resolution for every $k\in \NN$.
\item A (not necessarily finite) family of homogeneous ideals $\cF$ has \emph{linear products} if  for all $I_1,\dots, I_w\in \cF$ the product $I_1\cdots I_w$ has a linear resolution.
\end{enumerate}
\end{definition}

By \emph{resolution} we mean a graded free resolution, and we call such a resolution \emph{linear} if the matrices in the resolution have linear entries, of course except the matrix that contains the generators of the ideal. However, we assume that the ideal is generated by elements of constant degree. This terminology will be applied similarly to graded modules. 

Note that in (2) we have not demanded that the ideals are distinct, so, in particular, powers and products of powers of elements in $\cF$ are required to have linear resolutions.  

Given an ideal $I$ of $R$, we denote the associated Rees algebra by $R(I)$  and, similarly, for ideals $I_1,\dots, I_w$ of $R$ we denote the associated multi-Rees algebra by $R(I_1,\dots, I_w)$. If each  ideal $I_i$ is homogeneous and generated by elements of the same degree, say  $d_i$,  then $R(I_1,\dots, I_w)$ can be given the structure of a standard $\ZZ^{w+1}$-graded algebra. In Section \ref{ReesAlg} we will explain this in more detail.
  
The two notions introduced above can be characterized homologically.  In  \cite{R} and \cite{BCV}  it is proved that 
an ideal $I$ of $R$ has linear powers if and only if  $\reg_0 R(I)=0$.  In Theorem \ref{Tim} we extend this result by showing that  a family of ideals $\cF$ of $R$ has linear products if and only if for all  $I_1,\dots, I_w\in \cF$  one has  $\reg_0 R(I_1,\dots, I_w)=0$.  Here $\reg_0$ refers to the Castelnuovo-Mumford regularity computed according to the $\ZZ$-graded structure of the Rees algebra induced by inclusion of  $R$ in it, see Section \ref{PartReg} and Section \ref{ReesAlg} for the precise definitions and   for the proof of this statement. 

The prototype of  a family of ideals with linear products is  the following.

\begin{example} A monomial  ideal $I$ of $R$ is \emph{strongly stable} if  $I:(X_i)\supseteq  I:(X_j)$ for all $i,j$,  $1\leq i<j\leq n$. Strongly stable ideals are \emph{Borel fixed} (i.e., fixed under the $K$-algebra automorphisms of $R$ induced by upper triangular linear transformations). In characteristic $0$, every Borel fixed ideal is strongly stable. The regularity of a strongly stable ideal $I$ is the largest degree of a minimal generator  of $I$. Hence  a strongly stable ideal generated in a single degree has a linear resolution. Furthermore it is easy to see that  the product of strongly stable ideals is strongly stable. Summing up, the family
$$
\cF=\{ I : I \mbox{ is a strongly stable ideal generated in a single degree} \}
$$
has linear products. 
\end{example} 

The following example, discovered in the late nineties by the second author in collaboration with Emanuela De Negri, shows that the Rees algebra associated with  a strongly stable ideal need not be Koszul, normal or Cohen-Macaulay. 

\begin{example} In the polynomial ring $K[X,Y,Z]$ consider the smallest strongly stable ideal  $I$ that contains the three monomials $Y^6$, $X^2Y^2Z^2$,  $X^3Z^3$. The ideal $I$ is generated by 
\begin{multline*}
X^6, X^5Y, X^4Y^2, X^3Y^3, X^2Y^4, XY^5, Y^6, X^5Z, X^4YZ, X^3Y^2Z, X^2Y^3Z, X^4Z^2,\\
X^3YZ^2, X^2Y^2Z^2, X^3Z^3.
\end{multline*}
It has a non-quadratic,  non-normal and non-Cohen-Macaulay Rees algebra. Indeed, $R(I)$ is defined by $22$ relations of degree $(1,1)$, $72$ relations of degree $(0,2)$ and exactly one relation of degree $(0,3)$ corresponding to $(X^2Y^2Z^2)^3=(Y^6)(X^3Z^3)^2$ and its $h$-polynomial (the numerator of the Hilbert series) has negative coefficients. Therefore it is not Cohen-Macaulay, and by Hochster's theorem \cite[Th.~6.3.5]{BH} it cannot be normal.
\end{example} 

On the other hand, for a  \emph{principal strongly stable} ideal, i.e., the smallest strongly stable ideal containing a given monomial, the situation is much better. Say $u=X_1^{a_1}\cdots X_n^{a_n}$ is a monomial of $R$ and $I(u)$ is the smallest strongly stable ideal containing $u$. Then 
\begin{equation}
I(u)=\prod_{i=1}^n (X_1,\dots, X_i)^{a_i}=\Sect_{i=1}^n (X_1,\dots, X_i)^{b_i}, \qquad b_i=\sum_{j=1}^i a_j.\label{princ}
\end{equation}
Since the powers of an ideal generated by variables are primary and hence integrally closed, the  primary decomposition  formula \eqref{princ} implies right away that the ideal $I(u)$ is integrally closed.  It is an easy consequence of \eqref{princ} that for every pair of monomials $u_1,u_2$  one has 
$$I(u_1)I(u_2)=I(u_1u_2).$$ 
It follows that products of principal strongly stable ideals are integrally closed.  Hence
the multi-Rees algebra $R(I(u_1),\dots, I(u_w))$ associated with principal strongly stable ideals $I(u_1),\dots,I(u_w)$  is normal, which implies that it is also Cohen-Macaulay. 

Furthermore De Negri \cite{DN} proved that  the fiber  ring of $R(I(u_1))$ is defined by a Gr\"obner basis of quadrics and very likely  a similar statement  is  true for the  multi-fiber  ring of  the multi-Rees algebra  $R(I(u_1),\dots, I(u_w))$. 

Let us formalize the properties of the primary decomposition that we have observed for principal strongly stable ideals:

\begin{definition}\leavevmode
\begin{enumerate}
\item An ideal $I\subseteq R$ is of \emph{P-adically closed} if 
$$
I=\bigcap_{P\in\Spec R} P^{(v_P(I))}
$$
where $v_P$ denotes the $P$-adic valuation: $v_P(I)=\max\{k\ge 0: P^{(k)}\supseteq I\}$.
\item The family $\cF$ of ideals has the \emph{multiplicative intersection property} if for every product $J=I_1\cdots I_w$ of ideals $I_i\in\cF$ one has
$$
J=\bigcap_{P\in\cF\cap\Spec R} P^{(v_P(J))}.
$$
\end{enumerate}
\end{definition}

Note that a P-adically closed  ideal is integrally closed. Therefore all products considered in (2) are integrally closed. The multiplicative intersection property does not necessarily imply that the prime ideals in $\cF$ have primary powers, but this is obviously the case if $v_P(Q)=1$ whenever $Q\subset P$  are prime ideals in the family.

Our goal is to present three families $\cF$ of ideals that generalize the family of principal strongly stable ideals in different directions and have some important features in common: 
\begin{enumerate}
\item  $\cF$ has linear products;
\item $\cF$ has the multiplicative intersection property;
\item the multi-Rees algebras associated to ideals in the family have good homological and arithmetical  properties and defining equations (conjecturally) of low degrees. 
\end{enumerate}
These families are 
\begin{enumerate}
\item[(a)] polymatroidal ideals (Section \ref{sectpolym});
\item[(b)] ideals of linear forms (Section \ref{sectplf});
\item[(c)] Borel fixed ideals of maximal minors (Section \ref{sectneid}).
\end{enumerate} 
Each family has specific properties that will be discussed in detail. 

\section{Partial regularity}\label{PartReg}

Let $R$ be a $\ZZ^r$-graded ring. The degree of an element $x$ of $R$ is denoted by $\deg x$, and if we speak of the degree of an element it is assumed that the element is homogeneous. In the following it will be important to consider the partial $\ZZ$-degrees defined by the $\ZZ^r$-grading: by $\deg_i x$ we denote the $i$-th coordinate of $\deg x$, and speak of $\deg_i$ as the \emph{$i$-degree}. The same terminology applies to $\ZZ^r$-graded $R$-modules. One calls $R$ \emph{a standard $\ZZ^r$-graded} $K$-algebra if $R_0$, its component of degree $0\in\ZZ^r$, is the field $K$ and $R$ is generated as a $K$-algebra by homogeneous elements whose degree is one of the unit vectors in $\ZZ^r$. The ideal generated by the elements of nonzero degree is denoted $\mm_R$ or simply by $\mm$. It is a maximal ideal, and $K$ is identified with $R/\mm_R$ as an $R$-module.

If $n=1$, then no index is used for degree or any magnitude derived from it. Moreover, if it is convenient, we will label the coordinates of $\ZZ^r$ by $0,\dots,r-1$.

For any finitely generated graded module $M$ over a standard $\ZZ^r$-graded polynomial ring $R$ we can define \emph{(partial) Castelnuovo-Mumford regularities} $\reg_i(M)$, $i=1,\dots,r$. In the bigraded case they have been introduced by Aramova, Crona and De Negri \cite{ACD}. First we set
$$
\sup\nolimits_iM=\sup\{\deg_i x: x\in M\}.
$$
Next, let
$$
t_{ik}(M)=\sup\nolimits_i \Tor_k^R(K,M).
$$
Since the $\Tor$-modules are finite dimensional vector spaces, $t_{ik}(M)<\infty$ for all $i$ and $k$. Moreover, $t_{ik}(M)=-\infty$ for all $k>\dim R$ by Hilbert's syzygy theorem. Now we can define
$$
\reg_i^R(M)=\reg_i(M)=\sup_k \{t_{ik}(M)-k\}.
$$
The $\Tor$-modules have a graded structure since $M$ has a minimal graded free resolution\discuss{cF changed to cL}
$$
\cL:\ 0\to \Dirsum_{g\in\ZZ^r} R(-g)^{\beta_{pg}}\to \dots\to \Dirsum_{g\in\ZZ^r} R(-g)^{\beta_{0g}}\to M\to 0 ,\qquad p=\projdim M.
$$
Then $t_{ik}(M)$ is the maximum $i$-th coordinate of the shifts $g$ for which $\beta_kg\neq 0$. The $\beta_{kg}$ are called the \emph{$k$-th graded Betti numbers} of $M$.

As in the case of $\ZZ$-gradings one can compute partial regularities from local cohomology.  

\begin{theorem}\label {EG}
Let $R$ be a standard $\ZZ^r$-graded polynomial ring and $M$ a finitely graded $R$-module. Then
$$
\reg_i(M)=\sup_k\{\sup\nolimits_i H_\mm^k(M)+k\}
$$
for all $i$. 
\end{theorem}

For the theorem to make sense, one needs at least that the local cohomology with support in $\mm$ has a natural $\ZZ^r$-graded structure, and this indeed the case as one can see from its description by the \v{C}ech complex. In the $\ZZ$-graded case Theorem \ref{EG}\discuss{Explicit reference?} is due to Eisenbud and Goto \cite{EG}. For the proof of Theorem \ref{EG} one can follow \cite[p.~169]{BH}. The crucial point is multigraded local duality. It can be derived from $\ZZ$-graded local cohomology since the $\ZZ$-graded components of the modules involved are direct sums of \emph{finitely} many multigraded components. Therefore it makes no difference whether one takes graded $\Hom(-,K)$ in the category of $\ZZ$-graded modules or the category of $\ZZ^r$-graded modules.

We are interested in ideals with linear resolutions and for them Theorem \ref{EG} has an obvious consequence:

\begin{corollary}\label{Hlinres}
Let $R$ be standard $\ZZ$-graded polynomial ring over a field $K$, and $I$ an ideal generated in degree $d$ with a linear resolution. Then $\sup H_{\mm}^k(R/I)\le d-1-k$ for all $k$.
\end{corollary}

In fact, if $I$ has a linear resolution, then $\reg R/I=d-1$.

An extremely useful consequence of Theorem \ref{EG} is that it allows change of rings in an easy way.

\begin{lemma}\label{change}
Let $R$ be a standard $\ZZ^r$-graded polynomial ring over the field $K$ and let $x$ be an element of degree $e_i$ for some $i$. Suppose that $x$ is either a nonzerodivisor on the $\ZZ^r$-graded finitely generated module $M$ or annihilates $M$, and let $S=R/(x)$. Then $\reg_i^R(M)=\reg_i^S(M/xM)$.
\end{lemma}

\begin{proof}
If $x$ is a nonzerodivisor, then one can simply argue by free resolutions since $\cL\tensor S$ is a minimal graded free $S$-resolution of $M/xM$ if $\cL$ is such an $R$-resolition of $M$. In the second case, in which of course $M/xM=M$, one uses Theorem \ref{EG} and the invariance of local cohomology under finite ring homomorphisms.	
\end{proof}

We need some auxiliary results. The behavior of $\reg_i$ along homogeneous short exact sequences $0\to U\to M\to N\to 0$ is captured by the inequalities
\begin{align*}
\reg_i(M) &\le \max\{\reg_i(U),\reg_i(N)\},\\
\reg_i(U) &\le \max\{\reg_i(M),\reg_i(N)+1\},\\
\reg_i(N) &\le \max\{\reg_i(U)-1,\reg_i(M)\}.
\end{align*}
that follow immediately from the long exact sequence of $\Tor$-modules.

Another very helpful result is the following lemma by Conca and Herzog \cite[Prop.~1.2]{CH}:

\begin{lemma}\label{reg}
Let $R$ be standard $\ZZ$-graded and $M$ be a finitely generated graded $R$-module. Suppose $x\in R$ a homogeneous element of degree $1$ such that $0:_M x$ has finite length. Set $a_{0}=\sup H_{\mm}^0(M)$. Then
$\reg(M)=\max \{ \reg(M/xM), a_{0}\}$.
\end{lemma}

\section{The multi-Rees algebra}\label{ReesAlg}

The natural object that allows us to study all products of the ideals $I_1,\dots,I_w$ in a ring $R$ simultaneously is the \emph{multi-Rees algebra} 
$$
\Rees=R(I_1,\dots,I_w)=R[I_1T_1,\dots,I_wT_w]=\Dirsum_{a\in \ZZ_+^w} I_1^{a_1}\cdots I_w^{a_w}T_1^{a_1}\cdots T_w^{a_w}\subseteq R[T_1,\dots,T_w]
$$
where $T_1,\dots,T_w$ are indeterminates  over $R$. As a shortcut we set
$$
I^a=I_1^{a_1}\cdots I_w^{a_w}\qquad\text{and}\qquad T^a=T_1^{a_1}\cdots T_w^{a_w}.
$$
If $R$ is a standard graded algebra over a field $K$, and each $I_i$ is a graded ideal generated in a single degree $d_i$, then $R(I_1,\dots,I_w)$ carries a natural standard $\ZZ^{w+1}$-grading. In fact let $e_0,\dots.e_w$ denote the elements of the standard basis of $\ZZ^{w+1}$. Then we identify $\deg x$ with $(\deg x)e_0$ for $x\in R$ and set $\deg T_i=-d_ie_0+ e_i$ for $i=1,\dots,w$. Evidently $\Rees$ is then generated over $K$ by its elements whose degree is one of $e_0,\dots,e_w$.

We want to consider $\Rees$ as a residue class ring of a standard $\ZZ^{w+1}$-graded polynomial ring $\cS$ over $K$. To this end we choose a system $f_{i1},\dots,f_{im_i}$ of degree $d_i$ generators of $I_i$ for $i=1,\dots,w$ and indeterminates $Z_{i1},\dots,Z_{im_i}$. Then we set 
$$\
\cS=R[Z_{ij}: i=1,\dots,w,\ j=1,\dots,m_i]
$$
and define $\Phi:\cS\to\Rees$ by the substitution
$$
\Phi|R=\id_R, \qquad \Phi(Z_{ij})=f_{ij}T_i.
$$

We generalize a theorem of Römer \cite{R}, following the simple proof of Herzog, Hibi and Zheng \cite{HHZ}, and complement it by its converse:

\begin{theorem}\label {Tim}
Let $R$ be a standard graded polynomial ring over the field $K$. The family $I_1,\dots,I_w$ of ideals in $R$ has linear products if and only if $\reg_0(R(I_1,\dots,I_w))=0$.
\end{theorem}

\begin{proof}
For a graded module $M$ over $\cS$ and $h\in\ZZ^w$ we set
$$
M_{(\bullet,h)}=\Dirsum_{j\in\ZZ} M_{(j,h)}
$$	
where $(j,h)\in\ZZ\times \ZZ^w=\ZZ^{w+1}$. Clearly $M_{(\bullet,h)}$ is an $R$-submodule of $M$. It is the degree $h$ homogeneous component of $M$ under the $\ZZ^w$-grading in which we ignore the $0$-th partial degree of our $\ZZ^{w+1}$-grading.

We apply the operator $(\bullet,h)$ to the whole free resolution $\cL$ of $M=\Rees$ and obtain the exact sequence
$$
\cL_{(\bullet,h)}:\ 0\to \Dirsum_{g\in\ZZ} \cS(-g)_{(\bullet,h)}^{\beta_{pg}}\to \dots\to \Dirsum_{g\in\ZZ} \cS(-g)_{(\bullet,h)}^{\beta_{0g}}\to \Rees_{(\bullet,h)}\to 0.
$$
Since the $R$-modules $\cS(-g)_{(\bullet,h)}$ are free over $R$, we obtain a graded (not necessarily minimal) free resolution of $\cS_{(\bullet,h)}$ over $R$. The rest of the proof of this implication is careful bookkeeping of shifts.

First, as an $R$-module,
$$
\Rees_{(\bullet,h)}=I^hT^h\cong I^h(d\cdot h), \qquad d\cdot h=\sum d_ih_i.
$$
Moreover, $\cS(-g)_{(\bullet,h)}$ is a free module over $R$ with basis elements in degree $g_0$ since the indeterminates $Z_{ij}$ have degree $0$ with respect to $\deg_0$.

Now suppose that $\reg_0(\Rees)=0$. Then $g_{0k}\le k$ for all $k$, and we see that\discuss{CORRECTION}
$$
\reg I^h(d\cdot h)=\reg_0(I^h T^h)=0,
$$
and so $\reg(I^h)=d\cdot h$ as desired.

The converse is proved by induction on $\dim R$. For $\dim R=0$ there is nothing to show. In preparing the induction step we first note that there is no restriction in assuming that the ground field $K$ is infinite. Next we use that there are only finitely many prime ideals in $R$ that are associated to any of the products $I^h$ (West \cite[Lemma~3.2]{W}). Therefore we can find a linear form $z\in R_1$ that is not contained in any associated prime ideal of any $I^h$ different from $\mm_R$. In other words, $I^h:z$ has finite length for all $h$.

Set $S=R/(z)$. It is again a standard graded polynomial ring (after a change of coordinates). Furthermore let $J_i=I_iS=(I_i+z)/(z)$. Then $(I^h+(z))/(z)=J^h$. Next we want to compute $\reg^S(S/J^h)=\reg^R(S/J^h)$. The hypotheses of Lemma \ref{reg} are satisfied for $M=R/I^h$. Since $M/z M=S/J^h$ we obtain
$$
d\cdot h-1=\reg^R(R/I^h)=\max(\reg^R(S/J^h,\sup H_\mm^0(R/I^h)).
$$

Unless we are in the trivial case $J^h=0$, we have $\reg^S(S/J^h)=\reg^R(S/J^h) \ge d\cdot h-1$ anyway. Thus $\reg^S(S/J^h)=d\cdot h-1$, and therefore $S/J^h$ has a linear resolution over $S$. This allows us to apply the induction hypothesis to the family $J_1,\dots,J_w$, which appears in the exact sequence
$$
0\to W \to \Rees/z\Rees \to S(J_1,\dots,J_w)\to 0.
$$
In fact, by induction we conclude that $\reg_0 (S(J_1,\dots,J_w))=0$ over its representing polynomial ring, namely $\cS/(z)$. But because of Lemma \ref{change} we can measure $\reg_0$ over $\cS$ as well. Moreover, $\reg_0 \Rees/z\Rees=\reg_0 \Rees$.

It view of the behavior of $\reg_0$ along short exact sequences, it remains to show that $\reg_0 W=0$. As an $R$-module, $W$ splits into its $\ZZ^w$-graded components $W_h=W_h'T^h$ where
$$
W_h'=(I^h\cap (z))/z I^h=z(I^h:z)/z I^h.
$$
We claim that $W_h'$ is concentrated in degree $d\cdot h$, and therefore annihilated by $\mm_R$. Since it is a subquotient of $I^h$, it cannot have nonzero elements in degree $<d\cdot h$. On the other hand, the description $W_h'=z(I^h:z)/z I^h$ shows that the degree of the elements in $W$ is bounded above by $1+\sup (I^h:z)/I^h$. By the choice of $z$, this module is contained in $H_{\mm_R}^0(R/I^h)$ whose degrees are bounded by $d\cdot h-1$ because of Corollary \ref{Hlinres}. Multiplication by $T^h$ moves $W_h'$ into degree $0$ with respect to $\deg_0$ in $\Rees$.
 
Recombining the $W_h$ to $W$, we see that $W$ is generated as an $\cS$-module by elements of $0$-degree $0$. Moreover, it is annihilated by $\mm_R\cS$, and therefore a module over the polynomial ring $\cS'=\cS/\mm_R\cS$ on which the $0$-degree of $\cS$ vanishes. Hence the $0$-regularity of an $\cS'$-module $M$ is the maximum of $\deg_0 x$ where $x$ varies over a minimal system of generators of $M$. But, as observed above, $W$ is generated in $0$-degree $0$.
\end{proof}

\begin{remark}
(a) The first part of the proof actually shows that $\reg(I^h)\le d\cdot h +\reg_0(\Rees)$ for every $h$, as stated by Römer \cite[Thm.~5.3]{R} \discuss{added: single ideal} for a single ideal. It seems to be unknown whether there always exists at least one exponent $h$ for which  $\reg(I^h)= d\cdot h +\reg_0(\Rees)$. By virtue of  Theorem \ref{Tim}  this is the case if  $\reg_0(\Rees)\leq 1$.

(b) One can show that  the Betti number of the ideals $J^h$ over $S$ (notation as in the proof) are determined by those of the ideal $I^h$ over $R$. In the case of the powers of a single ideal this has been proved in \cite[Lemma~2.4]{BCV}.

(c) An object naturally associated with the multi-Rees algebra is the \emph{multi-fiber ring} 
$$
F(I_1,\dots,I_w)=R(I_1,\dots,I_w)/\mm_R R(I_1,\dots,I_w).
$$
If each of the ideals $I_i$ is generated in a single degree, then  $F(I_1,\dots,I_w)$ can be identified with the $K$-subalgebra of $R(I_1,\dots,I_w)$ that is generated by the elements of degrees $e_1,\dots,e_w$. As such, it is a retract of $R(I_1,\dots,I_w)$, and since the ideals of interest in this paper are generated in a single degree, we will always consider the multi-fiber ring as a retract of the multi-Rees algebra in the following.

 An ideal $I\subseteq R$ that is generated by the elements $f_1,\dots,f_m$ of constant degree is said to be of \emph{fiber type} if the defining ideal of the Rees algebra $R(I)$ is generated by polynomials that are either 
(1)   linear in the indeterminates representing the generators $f_iT$ of the Rees algebra (and therefore are relations\discuss{Should we add this?} of the symmetric algebra), or
(2) belong to the definining ideal of $F(I)$. One can immediately generalize this notion and speak of a family $I_1,\dots,I_w$ of \emph{multi-fiber type}. In the situation of the theorem it follows that $J_1,\dots,J_w$ is of multi-fiber type if $I_1,\dots,I_w$ has this property. 
\end{remark}

More generally than in Theorem \ref{Tim} one can consider arbitrary standard $\ZZ^r$-graded $K$-algebras. For them Blum \cite{B} has proved an interesting result. Note that a standard $\ZZ^r$-graded $K$-algebra is a standard $\ZZ$-graded $K$-algebra in a natural way, namely with respect to the total degree, the sum of the partial degrees $\deg_i$, $i=1,\dots,r$. Therefore it makes sense to discuss properties of a standard $\ZZ$-graded algebra in the standard $\ZZ^r$-graded case.

\begin{theorem}\label {Blum}
Suppose the standard $\ZZ^r$-graded algebra $R$ is a Koszul ring. Furthermore let $a,b\in \ZZ_+^n$ and consider the module
$$
M_{(a,b)} =\Dirsum _{k=0}^\infty R_{ka+b}
$$
over the ``diagonal'' subalgebra
$$
R_{(a)} = \Dirsum _{k=0}^\infty R_{ka}.
$$
Then $R_{(a)}$ is a standard graded $K$-algebra with $(R_{(a)})_k=R_{ka}$, and $M_{(a,b)}$ has a (not necessarily\discuss{Added} finite) linear resolution over it.
\end{theorem}

Blum states this theorem only in the bigraded case, but remarks himself that it can easily be generalized to the multigraded case. Let us sketch a slightly simplified version of his proof for the general case. Let $\mm$ be the ideal generated by all elements of total degree $1$. The ``degree selection functor'' that takes the direct sum of the $\ZZ^r$-graded components whose degrees belong to a subset of $\ZZ^r$ is exact. There is nothing to prove for $M_{(a,b)}$ if $b=0$. If $b\neq0$, then the degree selection functor for the subset $\{ka+b: k\ge 0\}$ cuts out an exact complex of $R_{(a)}$-modules from the resolution of $\mm$, which, by hypothesis on $R$, is linear over $R$. The modules occurring in it are not necessarily free over $R_{(a)}$, but they are under control, and an involved induction allows one to bound the shifts in their resolutions. These bounds in their turn imply that $M_{(a,b)}$ has a linear resolution over $R_{(a)}$. 

We highlight the following  special case of Theorem \ref{Blum}:  

\begin{theorem}\label {Blum-special} 
Let $R=K[X_1,\dots, X_n]$   and $I_1,\dots,I_w$ ideals of $R$ such that $R(I_1,\dots,I_w)$ is Koszul. Then the family $I_1,\dots,I_w$ has linear  products.
\end{theorem}

\begin{proof} Note that  $R$ is the diagonal subalgebra for $a=e_0$ of  $R(I_1,\dots,I_w)$  and one obtains $I^hT^h$ as a module over it for $b=(0,h)$.
\end{proof} 

\begin{remark} (a) In \cite[Ex.~2.6]{BCV} we give an example of a monomial ideal $I$ with linear powers that is not of fiber type and whose Rees algebra is not Koszul. But even linear powers and fiber type together do not imply that the Rees algebra is Koszul: see \cite[Ex.~2.7]{BCV}.
	
(b) For the ``if'' part of Theorem \ref{Tim} and for Theorem \ref{Blum} one can replace the polynomial base ring $R$ by a Koszul ring. This generalization allows relative versions in the following sense: if $R(I,J)$ is Koszul, then $R(I)$ is Koszul because retracts of Koszul rings are Koszul; since $R(I,J)=(R(I))(J) $, one obtains that $JR(I)$ has a linear resolution over $R(I)$.
\end{remark}

\section{Initial ideals and initial algebras}\label {initial}

We start again from a standard graded polynomial ring $R=K[X_1,\dots,X_n]$. It is a familiar technique to compare an ideal $J$ to its initial ideal $\ini_<(J)$ for a monomial order $<$ on $R$. The initial ideal is generated by monomials and therefore amenable to combinatorial methods. Homological properties like being Cohen-Macaulay or Gorenstein or enumerative data like the Hilbert series (if $I$ is homogeneous) can be transferred from $R/\ini_<(J)$ to $R/J$, and for others, like regularity, the value of $R/\ini_<(J)$ bounds that of $R/J$.

Let us formulate a criterion that will allow us to apply Theorem \ref{Tim} or even Theorem \ref{Blum} in order to conclude that a family of ideals has linear products. We use the presentation of the multi-Rees algebra $\Rees=R(I_1,\dots,I_w)$ as a residue class ring of a polynomial ring $\cS$ that has been introduced above Theorem \ref{Tim}.

\begin{theorem}
Let $\cI=\Ker\Phi$  be defining ideal of $\Rees$ as a residue class ring of the polynomial ring $\cS$, and let $\prec$ be a monomial order on $\cS$. Let $G$ be the minimal set of generators of the monomial ideal $\ini_\prec(\cI)$. 
\begin{enumerate}
\item If every element of $G$ is at most linear in the indeterminates $X_1,\dots,X_n$ of $R$, then $\reg_0(\Rees)=0$, and $I_1,\dots,I_w$ has linear products.

\item If $G$ consists of quadratic monomials, then $\reg_0(\Rees)=0$, $\Rees$ is a Koszul algebra, and $I_1,\dots,I_w$ has linear products.
\end{enumerate}
\end{theorem}

\begin{proof}
For the first statement (observed in \cite{HHZ}) it is enough to show that $\reg_0(\cS/\ini_\prec(\cI))=0$. Then we obtain $\reg_0(\Rees)=0$ as well and can apply Theorem \ref{Tim}. In order to estimate the maximal shifts in a minimal free resolution of $\cS/\ini_\prec(\cI)$ it is sufficient to estimate them in an arbitrary free resolution, and such is given by the Taylor resolution. All its matrices are have entries that are at most linear in the variables $X_1,\dots,X_n$ of $R$.

If $G$ consists of quadratic monomials, then $\cS/\ini_\prec(\cI)$ is Koszul by a theorem of Fröberg, and Koszulness of $\cS/\ini_\prec(\cI)$ implies the Koszulness of $\Rees=\cS/\cI$. In its turn this implies that $I_1,\dots,I_w$ has linear powers by Theorem \ref{Blum-special}, and then we obtain $\reg_0(\Rees)=0$ by Theorem \ref{Tim}.

However, we can circumvent Theorem \ref{Blum-special}. Namely, if $G$ consists of quadratic monomials, then it is impossible that any of them is of degree $\ge 2$ in $X_1,\dots,X_n$. In this multigraded situation this would imply that $\cI$ contains a nonzero element of $R$, and this is evidently impossible.   	
\end{proof}

In the previous theorem we have replaced the defining ideal of the multi-Rees algebra by a monomial object, namely its initial ideal. It is often useful   
to replace the multi-Rees algebra by a monomial algebra. In the following $<$ denotes a monomial order on $R$. We are interested in the products $I^h=I_1^{h_1}\cdots I_w^{h_w}$. There is an obvious inclusion, namely
\begin{equation}
\ini(I_1)^{h_1}\cdots \ini(I_w)^{h_w}\subseteq\ini(I^h),\label {ini_incl_id}
\end{equation}
and it is an immediate question whether one has equality for all $h$. Let us introduce the notation $\ini(I)^h$ for the left hand side of the inclusion.

To bring the multi-Rees algebra $\Rees=R(I_1,\dots,I_w)$ into the play we extend the monomial order from $R$ to $R[T_1,\dots,T_w]$ in an arbitrary way. There are two natural initial objects, namely the initial subalgebra $\ini(\Rees)$ on the one side and the multi-Rees algebra
$$
\Rees_{\ini}=R(\ini(I_1),\dots,\ini(I_w))
$$
on the other. Since $\Rees$ is multigraded in the variables $T_1,\dots,T_w$, $\ini(\Rees)$ does not depend on the extension of $<$ to $R[T_1,\dots,T_w]$, and we have
\begin{equation}
\Rees_{\ini}\subseteq\ini(\Rees). \label {ini_incl_rees}
\end{equation}
Equality in \eqref{ini_incl_id} for all $h$ is equivalent to equality in \eqref{ini_incl_rees}. Whether equality holds is a special case of the question whether polynomials $f_1,\dots,f_m$ form a Sagbi basis of the algebra $A=K[f_1,\dots,f_m]$ they generate. By definition, this means that the initial monomials $\ini(f_1),\dots,\ini(f_m)$ generate the initial algebra $\ini(A)$. Similarly to the Buchberger criterion for ordinary Gröbner bases there is a lifting criterion for syzygies. In this case the syzygies are polynomial equations.

We choose a polynomial ring $S=K[Z_1,\dots,Z_n]$. Then the substitution
$$
\Phi(Z_i)=f_i,\qquad i=1,\dots,m.
$$
makes $A$ a residue class ring of $S$. Simultaneously we can consider the substitution
$$
\Psi(Z_i)=\ini(f_i), \qquad i=1,\dots,m.
$$
Roughly speaking, one has $\ini(A)=K[\ini(f_1),\dots,\ini(f_m)]$ if and only if every element of $\Ker\Psi$ can be lifted to an element of $\Ker\Phi$. More precisely:

\begin{theorem}\label {lift}
With the notation introduced, let $B$ be a set of binomials in $S$ that generate $\Ker \Psi$. Then the following are equivalent:
\begin{enumerate}
\item $\ini(K[f_1,\dots,f_m])=K[\ini(F_1),\dots,\ini(f_m)]$;

\item For every $b\in B$ there exist monomials $\mu_1,\dots,\mu_q\in S$ and coefficients $\lambda_1,\dots\lambda_q\in K$, $q\ge 0$, such that
\begin{equation}
b-\sum_{i=1}^q \lambda_i\mu_i\in \Ker\Phi\label {liftedB}
\end{equation}
and $\ini(\Phi(\mu_i))\le \ini(\Phi(b))$ for $i=1,\dots,q$.
\end{enumerate}	
Moreover, in this case, the  elements in \eqref{liftedB} generate $\Ker\Phi$. 
\end{theorem}

The criterion was established by Robbiano and Sweedler \cite{RS} in a somewhat individual terminology; in the form above one finds it in Conca, Herzog and Valla \cite[Prop.~1.1]{CHV}. For the reader who would expect the relation $<$ in (2): if $b=\nu_1-\nu_2$ with monomials $\nu_1$, $\nu_2$, then $\ini(\Phi(b))<\ini(\Phi(\nu_1))=\ini(\Phi(\nu_2))$, unless $\Phi(b)=0$. The last statement of Theorem \ref{lift} is \cite[Cor.~2.1]{CHV}.

A priori, there is no ``natural'' monomial order on $S$. But suppose we have a monomial order $\prec$ on $S$. Then we can use $\prec$ as the tiebreaker in lifting the monomial order on $R$ to $S$:
$$
\mu\prec_\Psi \nu\quad\iff\quad \Psi(\mu)<\Psi(\nu) \text{ or } \Psi(\mu)=\Psi(\nu),\ \mu\prec \nu.
$$

\begin{theorem}\label {GBlift}
Suppose that the generating set $B$ of Theorem \ref{lift}(2) is a Gröbner basis of  $\Ker\Psi$ with respect to $\prec$. Then the elements in \eqref{liftedB} are a Gröbner basis of $\Ker\Phi$ with respect to $\prec_\Psi$.
\end{theorem}

See Sturmfels \cite[Cor.~11.6]{Stu} or \cite[Cor.~2.2]{CHV}.

\section{Polymatroidal ideals} 
\label{sectpolym}

A monomial  ideal $I$ of $R=K[X_1,\dots, X_n]$ is \emph{polymatroidal} if  it is  is generated in a single degree, say $d$, and for all pairs $u=\prod_h X_h^{a_h} ,v=\prod_h X_h^{b_h}$ of monomials of degree $d$ in $I$ and for every $i$ such that $a_i>b_i$ there exists $j$ with $a_j<b_j$ and $X_j(u/X_i)\in I$. A square-free polymatroidal ideal it is said to be \emph{matroidal} because the underlying combinatorial object is a matroid.  Conca and Herzog (\cite{CH}) and Herzog and Hibi (\cite{HHbook}) proved:

\begin{theorem} 
Every polymatroidal ideal has a linear resolution, and the product of polymatroidal ideals is  polymatroidal.
\end{theorem}
  
Hence we may say that  the family
$$
\cF=\{ I :  I \mbox{ is polymatroidal}\}
$$
has linear products. 

Those polymatroidal ideals that are obtained as products of ideals of variables are called \emph{transversal}. 
For example,  $(X_1,X_2)(X_1,X_3)$ is a transversal polymatroidal ideal and $(X_1X_2, X_1X_3, X_2X_3)$  is polymatroidal, but not transversal. 

In \cite{HV} Herzog and Vladoiu proved the following theorem on primary decomposition.

\begin{theorem}
\label{primdecpolym}
The family of polymatroidal ideals has the multiplicative intersection property:
given a polymatroidal ideal $I$, one has  
$$
I=\Sect_P P^{v_P(I)}
$$ 
where the intersection is extended over all the monomial prime ideals $P$ (i.e., ideals generated by variables). 
\end{theorem}

They also proved that $v_P(I)$ can be characterized as the ``local" regularity of $I$ and $P$, that is the regularity of the ideal obtained from $I$ by substituting $1$  for the variables not in $P$. Of course one gets an irredundant primary decomposition by restricting the intersection to the ideals  $P\in \Ass(R/I)$. The problem of describing the associated primes of a polymatroidal ideal in combinatorial terms is discussed by Herzog, Rauf and Vladoiu in  \cite{HRV} where they proved:
 
\begin{theorem}  
Every polymatroidal ideal $I$ satisfies 
$$
\Ass(R/I^k)\subseteq  \Ass(R/I^{k+1})\qquad \mbox{for all } k>0. 
$$
Furthermore the equality
$$
\Ass(R/I)=\Ass(R/I^{k})\qquad \mbox{for all } k>0
$$
holds, provided  $I$ is  transversal. 
\end{theorem} 
The latter equality is not true for general polymatroidal ideals. Examples of polymatroidal ideals $I$ such that $\Ass(R/I)\neq  \Ass(R/I^{2})$ are given by  some polymatroidal ideals of Veronese type. For example, for $I=(X_1X_2, X_1X_3, X_2X_3)$  one has $(X_1,X_2,X_3)\in \Ass(R/I^2)\setminus \Ass(R/I)$. 

From the primary decomposition formula of Theorem  \ref{primdecpolym} it follows that a polymatroidal ideal is integrally closed. Since products of polymatroidal ideals are polymatroidal, it follows then that the multi-Rees  algebra $R(I_1,\dots,I_w)$ of polymatroidal ideals $I_1,\dots, I_w$  is normal and hence Cohen-Macaulay by virtue of Hochster's theorem \cite[Th.~6.3.5]{BH}.  The same it true for the fiber ring $F(I_1,\dots, I_w)$ because it is an algebra retract of $R(I_1,\dots,I_w)$. Therefore: 

\begin{theorem}
\label{CMpolym}
Let  $I_1,\dots, I_w$ be polymatroidal ideals. Then $R(I_1,\dots,I_w)$ and $F(I_1,\dots, I_w)$ are  Cohen-Macaulay and normal. 
\end{theorem} 

White's conjecture, in its original form, predicts that the fiber ring, called the \emph{base ring} of the matroid in this context, associated to a single matroidal ideal is defined by quadrics, more precisely by quadrics arising from exchange relations.  White's conjecture  has been extended to (the fiber rings of)  polymatroidal ideals by Herzog and Hibi in \cite{HH} who``did not escape from  the temptation" to ask also if  such a ring is Koszul and defined by a Gr\"obner basis of quadrics.  These conjectures are still open. The major progress toward a solution has been obtained by Laso\'n and Micha{\l}ek  \cite{LM}  who proved White's conjecture ``up to saturation" for matroids. Further  questions and potential  generalizations of White's conjecture refer to the  Rees algebra $R(I)$ of a matroidal (or polymatroidal) ideal $I$: is it defined by (a Gr\"obner basis) quadrics?  Is it Koszul?

Note however that the fiber ring of the multi-Rees  algebra associated to polymatroidal ideals need not be defined by quadrics. For example: 

\begin{example} 
\label{notquad}
For $I_1=(X_1,X_2)$, $I_2=(X_1,X_3)$,  $I_3=(X_2,X_3)$ the fiber ring of $R(I_1,I_2,I_3)$ is 
$$
K[T_1X_1,T_1X_2,T_2X_1,T_2X_3,T_3X_2,T_3X_3]
$$ 
and it is defined by a single cubic equation, namely  
$$
(T_1X_1)(T_2X_3)(T_3X_2)=(T_1X_2)(T_2X_1)(T_3X_3).
$$ 
\end{example} 

Nevertheless for a polymatroidal ideal $I$  Herzog,  Hibi and Vladoiu proved in \cite{HHV} that the Rees algebra $R(I)$ is of fiber type,  and it  might be true  that  multi-Rees  algebras $R(I_1,\dots, I_w)$ associated to polymatroidal ideals $I_1,\dots,I_w$ are of multi-fiber type.

\section{Products of ideals of linear forms} 
\label{sectplf}
Let $P_1,\dots, P_w$ be ideals of $R=K[X_1,\dots, X_n]$ generated by linear forms. Each $P_i$ is clearly a prime ideal with primary powers. 
One of the main results of  \cite{CH} is the following: 

\begin{theorem}  
The product $P_1\cdots P_w$ has a linear resolution. 
\end{theorem} 

Hence may say that  the family
$$\cF=\{ P :    P \mbox{ is generated by linear forms and  } P\neq 0\}$$
has linear products. The theorem is proved by induction on the number of variables. The inductive step  is based on an estimate of the $0$-th local cohomology of the corresponding quotient ring or, equivalently, on the saturation degree of the corresponding ideal. The latter is controlled by means of the following primary decomposition computed in \cite{CH}:

\begin{theorem} 
\label{PrimDecPLF}
The family of ideals generated by linear forms has the multiplicative intersection property. In other words, for every $P_1,\dots, P_w\in \cF$ and $I=\prod_{i=1}^w P_i$ one has: 
$$
I= \Sect_{P\in \cF}   P^{v_P(I)}.
$$ 
\end{theorem} 

Clearly one can restrict the intersection to the primes of the form 
$$
P_A=\sum_{i\in A} P_i
$$ for a non-empty subset $A$ of $\{1,\dots,w\}$.  Setting  
$$
\cP= \{ P :  P=P_A \mbox{ for some  non-empty subset  $A$ of } \{1,\dots,w\} \},
$$  
one gets 
$$
I= \Sect_{P\in \cP}   P^{v_P(I)}.
$$ 
This primary decomposition need not be irredundant.  So an important  question is whether a given $P\in \cP$ is  associated to $R/I$. 
Inspired by results in \cite{HRV}, we have a partial answer: 

\begin{lemma} 
With the notation above, let $P\in \cP$ and let $V=\{ i : P_i\subseteq  P\}$. 
\begin{enumerate}
\item Let $G_P$ be the graph with vertices $V$ and edges $\{i,j\}$ such that  $P_i \sect P_j$ contains a non-zero linear form. If $G_P$ is connected, then $P\in \Ass(R/I)$.

\item Assume that $P$ can be written as $P'+P''$ with $P', P''\in \cF $ such that $P'\sect P''$ contains no linear form and for every $i\in V$ one has either $P_i\subseteq P'$ or  $P_i\subseteq P''$. Then $P\notin \Ass(R/I)$.
\end{enumerate} 
\end{lemma} 

\begin{proof} 
(1) By localizing at $P$  we may right away assume that $P=\sum_{i=1}^w P_i=(X_1,\dots, X_n)$ and $V=\{1,\dots,w\}$. Since the graph $G_P$ is connected, we can take a spanning tree $T$ and for each edge $\{i,j\}$ in $T$ we may take a linear form $L_{ij}\in P_i\sect P_j$. The product $F=\prod_{(i,j)\in T} L_{ij}$ has degree $w-1$ and, by construction, $P_iF\subseteq I$ for all $i$. Hence $P\subseteq  I:F$. Since $F\notin I$ by degree reasons, it follows that $P=I:F$. 

(2) Again we may assume $P=\sum_{i=1}^w P_i =(X_1,\dots, X_n)$ and we may further assume that $P'=(X_1,\dots, X_m)$ and $P''=(X_{m+1}, \dots, X_n)$. We may also assume that $P_i\subseteq P'$ for $i=1,\dots,c$ and $P_i\subseteq P''$ for $i=c+1,\dots,w$ for some $m$ and $c$ such that $1\leq m<n$ and $1\leq c<w$.
Set $J=P_1\cdots P_c$ and $H=P_{c+1}\cdots P_w$. Then $I=JH=J\sect H$ because $J$ and $H$ are ideals in distinct variables. We may conclude that any associated prime of $I$ is either contained in $P'$ or in $P''$, and hence $P$ cannot be associated to $I$. 
\end{proof} 

When each of the $P_i$ is generated by indeterminates, then $I$ is a transversal polymatroid, and for a given $P\in \cP$ either (1) or (2) is satisfied. Hence we have, as a corollary, the following results of  \cite{HRV}. We state it in a slightly different form. 

\begin{corollary} Let  $I=P_1\cdots P_w$ with $P_i\in \cF$ generated by variables (i.e., $I$ is a transversal  polymatroidal ideal).  Then $P\in \cP$ is associated to $R/I$ if and only if the graph $G_P$ is connected. 
\end{corollary} 

But in general, for $I=P_1\cdots P_w$, a prime  $P\in \cP$ can be associated to $I$ even when $G_P$ is not connected: 

\begin{example} 
Let $R=K[X_1,\dots,X_4]$ and let  $P_1,P_2,P_3$  be ideals generated by $2$ general linear forms each and $I=P_1P_2P_3$.  Then $P=P_1+P_2+P_3=(X_1,\dots, X_4)$ is associated to $R/I$ and $G_P$ is not connected (it has no edges). That $P$ is associated to $R/I$ can be  proved by taking a non-zero quadric $q$ in the intersection $P_1\sect P_2 \sect P_3$ and checking that, by construction, $qP\subseteq I$. 
\end{example} 

The general question of whether a prime ideal $P\in \cP$ is associated to $R/I$ can be reduced by localization to the following: 

\begin{question} 
Let $P_1,\dots, P_w\in \cF$ and $I=P_1\cdots P_w$. Under which (possibly combinatorial) conditions on $P_1,\dots, P_w$ is $\sum_{i=1}^w  P_i$ associated to $R/I$? 
\end{question} 

Another interesting (and very much related)  question is  the description of the relationship between the associated primes of $I$ and those of its powers. 
We have: 

\begin{lemma} 
Let $I$ and $J$ be ideals that are products of elements in $\cF$.  Then 
$$
\Ass(R/I)\cup \Ass(R/J) \subseteq \Ass(R/IJ).
$$
In particular, 
$\Ass(R/I^h)\subseteq \Ass(R/I^{h+1})$ for all $h>0$. 
\end{lemma}

\begin{proof} Let $P_1,\dots, P_w, Q_1,\dots, Q_v \in \cF$  such that  $I=P_1\cdots P_w$ and $J=Q_1\cdots Q_v$.   
Let $P\in\Ass(R/I)$. We know that $P=\sum_{i\in A} P_i$ for a subset $A$ of $\{1,\dots, w\}$. Localizing at $P$ we may restrict our attention to the factors $P_i\subseteq  P$ and $Q_j\subseteq P$. Hence we assume that 
$P=\sum_{i=1}^w P_i=(X_1,\dots,X_n)$.  Let $f$ be a homogeneous element such that $P=I:(f)$. Since $I$ has a linear resolution, it coincides with its saturation from degree $w$ on.  Hence $f$ has degree $w-1$. Then $P=(IJ):(fJ)$ because $fJ\not\subseteq IJ$ by degree reasons. 
\end{proof} 

The main question here is the following: 

\begin{question} 
Let $P_1,\dots, P_w\in \cF$ and $I=P_1\cdots P_w$. Is it true that $\Ass(R/I)=\Ass(R/I^k)$ for every $k>0$? 
\end{question} 

We conclude the section with the following: 

\begin{theorem} 
\label{CMPLF}
Let $P_1,\dots, P_w\in \cF$. The multi-Rees  algebra $R(P_1,\dots, P_w)$ and its multi-fiber ring $F(P_1,\dots, P_w)$ are  normal and Cohen-Macaulay. Furthermore they are defined by Gr\"obner bases  of elements of degrees  bounded above by $\sum_{i=0}^w e_{i}  \in \ZZ^{w+1}$. 
\end{theorem} 

\begin{proof}  The  multiplicative intersection property implies that a product of elements in $\cF$ is integrally closed, and this entails the normality of $R(P_1,\dots, P_w)$. The multi-fiber ring is normal as well because it is an algebra retract of the Rees algebra. By construction,  $R(P_1,\dots, P_w)$ can be identified  with  $F(P_0,P_1,\dots, P_w)$ where $P_0=(X_1,\dots, X_n)$. Hence it is enough to prove the Cohen-Macaulay property of the multi-fiber ring $F=F(P_1,\dots, P_w)$. Note that  $F$ is a subring of the Segre product $R*S$ of $R$ with $S=K[T_1,\dots, T_w]$.  The defining ideal of $R*S$, i.e.,  the ideal of $2$-minors of a generic $m\times w$ matrix, has a square-free generic initial ideal with respect to the $\ZZ^w$-graded structure, as proved in \cite{C}. So it is a Cartwright-Sturmfels ideal, a notion  defined in Conca, De Negri and Gorla \cite{CDG2} that was inspired by result of Cartwright and Sturmfels  \cite{CS} and Conca, De Negri and Gorla \cite{CDG}. In \cite{CDG2} it is proved that eliminating variables from a Cartwright-Sturmfels ideal one gets a Cartwright-Sturmfels ideal. So $F$ itself is defined by a Cartwright-Sturmfels ideal. 
Such an ideal has a multiplicity-free multidegree. Hence we may use Brion's theorem \cite{B}  asserting that a multigraded prime ideal with  multiplicity-free multidegree defines a Cohen-Macaulay ring. Finally the statement about the degrees of Gr\"obner basis elements is a general fact about  Cartwright-Sturmfels ideals. 
\end{proof} 

As we have seen already in \ref{notquad}, we cannot expect that $R(P_1,\dots, P_w)$ and $F(P_1,\dots, P_w)$ are  defined by quadrics. But  the strategy developed in  \cite{C}  together with  \ref{CMPLF} implies: 

\begin{theorem}
Let $P_1,\dots, P_w\in \cF$ and $I=P_1\cdots P_w$. Then the fiber ring $F(I)$ is Koszul. 
\end{theorem}

\begin{remark}
The result on linear products of ideals generated by linear forms has been generalized by Derksen and Sidman \cite{DS}. Roughly speaking, 
they show that the regularity of an ideal that is constructed from ideals of linear forms by $D$ successive basic operations like products, 
intersections and sums is  bounded by $D$.
\end{remark} 

\section{Product of Borel fixed ideals of maximal minors} 
\label{sectneid}
Let $K$ be a field.  Let $X=(X_{ij})$ be the matrix of size $n\times n$   whose entries are the indeterminates of the polynomial ring $R=K[x_{ij}: 1\le i, j \le n]$. 
Let $t$ and $a$ be positive integers with that $t+a\leq n+1$ and set 
$$
X_t(a)=(X_{ij}: 1\le i \le t,\ a\le j \le n),
$$
and define the  \emph{northeast ideal} $I_t(a)$ associated to the pair $(t,a)$  to be the ideal generated by the $t$-minors, i.e., $t\times t$-subdeterminants,   of the matrix $X_t(a)$. Note that, by construction, $X_t(a)$ has size $t\times (n+1-t)$ and $t\leq (n+1-t)$. Hence $I_t(a)$ is the ideal of maximal minors of $X_t(a)$. 
There is a  natural action of $\GL_n(K)\times \GL_n(K)$ on $R$. Let $B_n(K)$ denote the subgroup   of lower triangular matrices in $\GL_n(K)$ and by $B'_n(K)$  the subgroup of upper  triangular matrices.  Note that the ideals $I_t(a)$ are fixed by the  action of the Borel group $B_m(K)\times B'_n(K)$. Hence they are  Borel fixed ideals of maximal minors.    

Let $<$ be the lexicographic term order on $R$ associated to the total order 
$$
X_{11}>X_{12}>\dots>X_{1n}>X_{21}>\dots>X_{nn}.
$$
Then the initial monomial of a $t\times t$ subdeterminant of $X$ is the product of its diagonal elements.
Therefore $<$ is a \emph{diagonal} monomial order. The statements below remain true if one replaces $<$ with another diagonal monomial order. 

Set 
$$
J_t(a)= (  X_{1b_1}\cdots X_{tb_t} :  a\leq b_1<\dots< b_t \leq n).
$$
It is the ideal generated by the initial monomials of the $t\times t$ minors in $I_t(a)$.

The ideal of maximal minors of a matrix of variables, such as  $X_t(a)$, is a prime ideal with primary  powers. Furthermore 
$$
\ini_<(I_t(a))=J_t(a).
$$ 
The main result of \cite{BC} is the following: 

\begin{theorem}  The families 
$$
\cF=\{  I_t(a) : \ \  t>0, \ \  a>0,   \ \ t+a\leq n+1 \}
$$
and
$$
\cF'=\{  J_t(a) : \ \  t>0, \ \  a>0,   \ \ t+a\leq n+1 \}
$$
have linear products.   

Furthermore, given $(t_1,a_1),\dots, (t_w,a_w)$, set $I_j=I_{t_j}(a_j)$ and $J_j=J_{t_j}(a_j)$. Then  
$$
\ini_<(R(I_1,\dots, I_w))= R(J_1,\dots, J_w),
$$
and both multi-Rees algebras $R(I_1,\dots, I_w)$ and $R(J_1,\dots, J_w)$ as well as their  multi-fiber rings $F(I_1,\dots, I_w)$ and $F(J_1,\dots, J_w)$ are Cohen-Macaulay, normal and defined by Gr\"obner bases of quadrics. 
\end{theorem} 

The proof of the theorem is based on the general strategy described in Section \ref{initial} and on the following decomposition formulas proved in \cite{BC}:

\begin{theorem}  
\label{primdecBmin}
For every $S=\{ (t_1,a_1),\dots, (t_w,a_w) \} $ set $I_S=\prod_{i=1}^w I_{t_i}(a_i)$ and $J_S=\prod_{i=1}^w J_{t_i}(a_i)$. Then
$$
I_S=\bigcap_{u,b}  I_u(b)^{e_{ub}(S)}
$$
and 
$$
J_S=\bigcap_{u,b}  J_u(b)^{e_{ub}(S)}
$$
where 
$$
e_{ub}(S)=  |\{i: b \le a_i \text{ and } u\le t_i \}|.
$$
\end{theorem}

The theorem allows one to define a certain normal form for elements in $I_S$. The passage to the normal form uses quadratic rewriting rules that represent Gröbner basis elements of the multi-Rees algebra.

The exponent $e_{ub}(S)$ can be characterized as well by the equalities 
$$
e_{ub}(S)= \max\{ j :  I \subseteq I_u(b)^j \} = \max\{ j :  J \subseteq J_u(b)^j \}.
$$
The theorem shows that the family $\cF$ of NE-ideals of maximal minors has the multiplicative intersection property since the representation of $I$ as an intersection is a primary decomposition: each ideal $I_u(b)$ is prime with primary powers. The representation of $J$ is not a primary decomposition. To get a primary decomposition for $J$ one uses  the fact that $J_u(b)$ is radical with decomposition 
$$
J_u(b)=\bigcap_{F\in F_{u,b} } P_F
$$
where $F_{u,b}$ denotes the set of facets of the simplicial complex associated to $J_u(b)$, and $P_F$ the prime ideal  associated to $F$. Moreover, 
$$
J_u(b)^k=\bigcap_{F\in F_{u,b} } P_F^k
$$
as proved in  \cite[Prop.~7.2]{BC3}. Therefore, while the family  $\cF'$  does not have the multiplicative intersection property, its members are nevertheless P-adically closed.

The primary decomposition of $I_S$ given in  \ref{primdecBmin} can be refined as follows: 

\begin{proposition}
\label{irred} 
Given $S=\{ (t_1,a_1),\dots, (t_w,a_w) \} $, let  $Y$ be the set of the elements $(t,a)\in \NN_{+}^2\setminus S $  such that there exists $(u,b)\in  \NN_{+}^2$ for which  $(t,b),(u,a)\in S$ and $t<u$, $a<b$. 
Then we have the following  primary decomposition:  
$$
I_S = \bigcap_{(v,c)\in S \cup Y}I_v(c)^{e_{vc}(S)}.
$$ 
Furthermore this decomposition is irredundant, provided all the points $(u,b)$ above can be taken so that $u+b\leq  n+1$.
\end{proposition}

Note that for a given $S$, the primary decomposition in \ref{irred} is irredundant if $n$ is sufficiently large. Therefore we obtain: 

\begin{corollary}
\label{stable} 
Given $S=\{ (t_1,a_1),\dots, (t_w,a_w) \} $ assume that $n$ is sufficiently large. Then 
$$
\Ass(R/I_S)=\Ass(R/I_S^k)
$$ 
for all $k>0$. 
\end{corollary} 

In some cases the equality stated in Corollary  \ref{stable} holds true also for small values of $n$. 

\begin{remark}
(a) There exist well-known families of ideals that have the multiplicative intersection property, but lack linear products. In characteristic $0$ this holds for the ideals $I_t(X)$, $t=0,\dots,m$, that are generated by all $t$-minors of the full matrix $X$ if $m>2$. See Bruns and Vetter \cite[Th.~10.9]{BV}.

(b) In \cite{BCV} the authors and Varbaro proved that the ideal of maximal minors of a matrix of linear forms has linear powers under certain conditions that are significantly weaker than ``full'' genericity. The techniques applied in \cite{BCV} are quite different from those on which the results of this note rely. 
\end{remark}

\end{document}